\documentclass[10pt]{amsart}

\numberwithin{equation}{section}
\usepackage[a4paper, left=3.2cm, right=3.2cm]{geometry}
\usepackage{xypic}
\usepackage[english]{babel}

\usepackage{amsmath}
\usepackage{amsfonts}
\usepackage{amssymb}
\usepackage{amscd}
\usepackage{amsthm}
\usepackage{color}
\usepackage{xcolor}
\usepackage{verbatim}
\usepackage[colorlinks, linktocpage, citecolor = blue, linkcolor = blue]{hyperref}
\usepackage{tikz}	

\usetikzlibrary{matrix}					%geometric/algebraic description.
\usetikzlibrary{arrows, shapes, snakes, 
		       automata, backgrounds, 
		       petri, topaths}				
\newtheorem{theorem}{Theorem}[section]

\newtheorem{lemma}[theorem]{Lemma}
\newtheorem{proposition}[theorem]{Proposition}

\theoremstyle{definition}
\newtheorem{definition}[theorem]{Definition}

\newtheorem{remark}[theorem]{Remark}

\theoremstyle{remark}

\usepackage{amscd, amssymb}
\newcommand\mynote[1]{\marginpar{\ \\ \small \tt #1}}
\newcommand\bel[1]{{\mynote{#1}}\begin{equation}\label{#1}}
\newcommand\mylabel[1]{\label{#1}}

\newcommand{\ZZ}{\mathbb{Z}}

\renewcommand{\AA}{\mathbb{A}}
\newcommand{\GG}{\mathbb{G}}

\newcommand  {\shO}     {\mathcal{O}}

\newcommand  {\Br}      {\operatorname{Br}}

\newcommand  {\et}      {{\text{\rm \'{e}t}}}

\newcommand  {\GL}      {\operatorname{GL}}

\newcommand  {\lra}     {\longrightarrow}

\newcommand  {\Pic}     {\operatorname{Pic}}
\newcommand  {\PGL}     {\operatorname{PGL}}

\newcommand  {\ra}      {\rightarrow}

\newcommand  {\Spec}    {\operatorname{Spec}}

\newcommand {\gal}      {{\operatorname{Gal}}}

\def\mydate{\number\day\space\ifcase\month \or January\or February\or March\or
April\or May\or June\or July\or
August\or September\or October\or November\or December\fi \space\number\year}
\begin{document}

\title[ Azumaya algebras over unramified extensions of function fields ]
      {Azumaya algebras over unramified extensions of function fields}

\author[Mohammed Moutand]{Mohammed Moutand}
\address{ Moulay Ismail University, 
 Department of mathematics, 
Faculty of sciences,  Mekn\`es,  B.P. 11201 Zitoune,  Mekn\`es,  Morocco.}
\curraddr{}
\email{m.moutand@edu.umi.ac.ma}
%\subjclass[2010]{14F22, 14F35}

%\subjclass{14F22,  20J06,  32J15}

%\dedicatory{Final version,  29 January 2020} 

\begin{abstract}
 Let $X$ be a smooth  variety over a field $K$ with function field $K(X)$. Using  the interpretation of the torsion part of the \'etale cohomology group  $ H_\et^2(K(X),  \GG_{m})$  in terms of Milnor-Quillen algebraic $K$-group $K_2(K(X))$, we prove that
under mild conditions on the norm maps along  unramified  extensions of $K(X)$  over $X$, there exist  cohomological Brauer classes  in $   H_\et^2(X,  \GG_{m})$ that are representable by  Azumaya algebras on $X$. Theses conditions  are almost  satisfied   in the case of number fields, providing then,  a partial answer on a question  of  Grothendieck.
\end{abstract}

\maketitle

\section{Introduction}
Let $X$ be a locally noetherian scheme, and let  $\Br'(X) := H_{\et}^2(X, \GG_{m})_{\rm tors}$  be the  cohomological Brauer group of $X$.  In \cite{GR1} Grothendieck  raised the question of  whether any Brauer class $\alpha \in \Br'(X)$ is representable by an Azumaya algebra on $X$, that is when $\alpha$ lies in the Brauer group $\Br(X)$. The answer to  this question  is known to be positive for some special cases, among which: regular schemes of dimension $\leq 2$ (Grothendieck \cite[II]{GR1}), affine schemes and separated unions of two affine schemes (Gabber \cite{GBBR}), abelian varieties (Hoobler \cite{HOO1} and Berkovich \cite{BERK}) and  separated normal algebraic surfaces (Schr\"oer \cite{SCH1}). The most general solved  case, which  includes quasi-projective varieties and, in particular, abelian varieties, is due to Gabber (unpublished), where he proved the conjecture for any scheme admitting an ample invertible  sheaf (see \cite{DJNG} or \cite[Section 4.2]{colliot2} for an alternative proof by de Jong). In \cite{MTD,MTD1}, the author gave positive answers for some  classes of schemes in terms of the \'etale homotopy type and pro-finite \'etale covers. For counterexamples, a non-separated normal surface for which \(\operatorname{Br}(X) \neq \operatorname{Br}'(X)\) was constructed in \cite{EHKV} using stack theory (see also \cite{Ber} for a simple alternative proof).

In this note,  we investigate the question for  varieties over a field $K$ under a  condition on the norm map of algebraic $K_2$-groups of finite extensions of $K$, which is almost satisfied in the number field case.

Let us first fix some notations: For a field $F$ and an integer $n \geq 0$,  let $K^M_n(F)$ denote the $n$-th Milnor algebraic $K$-group of $F$ (cf. \cite[Chapter 7]{giltmas}). For $n=2$, this group agrees with the second Quillen algebraic $K$-group $K_2(F)$ in the sense of \cite{webl}.

Now fix a  field $K$. Let $X$ be a  regular integral  scheme of finite type over  $K$, with function field $K(X)$, and $p$ an integer prime to $char(K)$. Let   $\alpha  \in \Br'(X)_p$ be a  $p$-torsion Brauer class. It can be identified with its image in  $H^2_\et(K(X),\mu_{p})$  (see Lemma \ref{lift}). If moreover $K$  contains a primitive $p$-th root of unity,   we  can again  identify $\beta$ with its image in $K_2(K(X))_p:=K_2(K(X))/pK_2(K(X))$ by the norm residue isomorphism (Theorem \ref{norm}). We express this identification through the following  composition:

\begin{equation}\label{identif}
\Br'(X)_p \hookrightarrow   H^2_\et(K(X),\mu_{p}) \overset{\simeq}\lra K_2(K(X))_p.
\end{equation}
\\
\\

The aim of this note is to prove the following:

\begin{theorem}\label{mainthm}

Let $K$ be a field that  contains a primitive $p$-th root of unity $\zeta_{p}$, with $p \neq char(K)$. Let $X$ be a smooth  variety over  $K$, and let  $\alpha  \in \Br'(X)_p$  be a $p$-torsion cohomological  Brauer class. Suppose that there is a finite  Galois extension $L/K$ such that:

\begin{itemize}

\item[(i)]   $p$ divides  $[L:K]$,
\item[(ii)] the norm map $N_{L/K}: K_2(L) \ra K_2(K) $ is surjective,
\item[(iii)] the image of  $\alpha$  in $K_2(K(X))$ (under the  identifications \eqref{identif}) comes from $K_2(K)$.
\end{itemize}
Then the class  $\alpha$ lies in $\Br(X)$.
\end{theorem}

As we will see in the last section, the proposed  assumptions on the ground field $K$ and the extension $L$, can be adjusted to get the  same result with alternative arguments.

Regarding condition $(ii)$, Theorem \ref{mainthm} works almost perfectly when  $K$ is a number field. This is primarily due to the local property of norm maps, which is  independent of finite extensions of $K$ (Theorem \ref{normfield}), along with the descent of Azumaya algebras along base change (Lemma \ref{frobinus}). Note that if $K$ does not contain a primitive $p$-th root of unity, we have only a natural (non necessarily bijective) morphism $ H^2_\et(K(X),\mu_{p}) \lra K_2(K(X))_p$. We prove then, the following theorem.
\begin{theorem}\label{mainth2}
Let $X$ be a smooth  variety over a number field $K$. Let  $\alpha \in \Br'(X)_p$ be a $p$-torsion cohomological  Brauer class. Suppose that the image of $\alpha$ in $K_2(K(X))$   comes from $K_2(K)$, then  $\alpha$ lies in $\Br(X)$.
\end{theorem}

As noted, condition $(iii)$ of Theorem \ref{mainthm} is the only one used in Theorem  \ref{mainth2}, and it relies solely on the $K$-theory of the function field  $K(X)$. Moreover, to verify this condition, we are free to replace the field $K$ by any Galois extension, thanks to a result of Colliot-Th\'el\`ene-Suslin(Theorem \ref{colliogalois}).

The proof of  Theorem \ref{mainthm} is essentially based on a result of  Gabber (Lemma \ref{galois}) on the descent of Azumaya algebras along finite locally free morphisms.  For this,  a splitting finite \'etale cover of the Brauer class $\alpha$, will be provided by the extension $L$, via the  unramified extensions of the  function field $K(X)$ over  $X$ (Definition \ref{unramified}).

Throughout this paper we consider the following notations:

\begin{itemize}
 \item   A variety  over a field $K$ is a  separated integral scheme of finite type over $K$.
  \item  $\GG_{m}$:  the \'etale sheaf of multiplicative groups.
\item $\mu_{n}$: the \'etale sheaf of $n$-th roots of unity.
\item For a scheme $X$, we denote $char(X)$ as the set  of residue characteristics of $X$.

 \item For an abelian group $A$,  notations  ${}_n\! A$ and $A_n$  will stand for the kernel and the cokernel of the homomorphism $a \ra n.a$ of $A$.
\item For a  scheme $X$ over a field  $K$, and $L/  K$ a field extension, $X_L= X \times_K L$ denotes the base change of $X$ to $L$.
\item For a given field $F$,   we will work with the  \'etale cohomology of $\Spec(F)$ with values in the \'etale sheaf $\GG_{m}$ and   $\mu_{n}$,  which agrees  with Galois  cohomology of $F$ with values in the multiplicative group and the group of  $n$-th roots of unity  respectively.

\end{itemize}

\section{Preliminaries}

We will review some preliminary notations  and results to set the stage for proving the main results of this paper.

\subsection{Brauer group and Brauer map} All the materials presented in this subsection can be found  with proofs and details in  \cite{GR1} or \cite[Chapter IV]{MIL} (see also \cite[Chapter 3]{colliot2}).

 Let $X$ be a  scheme. The Brauer group $\Br(X)$ of  $X$ is  defined as the set of classes of Azumaya $\shO_X$-algebras on $X$ modulo similarity equivalence. For a given integer $r \geq 2$, any  Azumaya algebra of rank $r^2$ corresponds isomorphically  to an element of the  \v Cech cohomology set $\check{H}_{\et}^1(X,   \PGL_r(\mathcal{O}_X))$. On the other hand, there is an exact sequence of \'etale sheaves on $X$  
$$ 1 \longrightarrow  \mathbb{G}_{m}    \longrightarrow  \GL_r(\mathcal{O}_X)  \lra  \PGL_r(\mathcal{O}_X) \longrightarrow 1,$$ 
which gives, via  non abelian cohomology,  boundary maps of \v Cech cohomology sets
 $$\check{H}_{\et}^1(X,   \PGL_r(\mathcal{O}_X))    \overset{\delta_r}\longrightarrow  \check{H}_{\et}^2(X, \mathbb{G}_{m}).$$
The composition with  the \'etale cohomology  induces an injective homomorphism of groups
$$\delta: \Br(X) \lra \Br'(X),$$
called the Brauer map, where $\Br'(X) := H_{\et}^2(X, \GG_{m})_{\rm tors}$ is the torsion part of the \'etale cohomology group  $H_{\et}^2(X, \GG_{m})$, called the cohomological Brauer group.  For a quasi-compact scheme $X$, Grothendieck asked in \cite{GR1} whether the map $\delta$ is surjective, that is when a Brauer class $\alpha \in \Br'(X)$ comes from an Azumaya algebra in $\Br(X)$. Or   in conventional terms; when $\Br(X) = \Br'(X)$?. We refer to \cite{MTD} for a brief introductory exposition on the various methods used to solve this question, as well as for a  list of the most recent  solved cases.  For our case in this note, we will use the following  technical criterion of Gabber. 

\begin{lemma}\cite[page 165, Lemma 4]{GBBR}\label{galois}
\mylabel{galois}
 Let  $X$ be  a locally noetherian  scheme, and $\alpha \in  \Br'(X)$. If there exists a finite locally free  morphism $f: Y \ra X$ such that $\alpha_{|Y} \in  \Br(Y)$ (where  $\alpha_{|Y}$ is the image of $\alpha$ under the pulback morphism $f^*:\Br'(X) \ra \Br'(Y)$), then $\alpha \in  \Br(X)$.
\end{lemma}

\begin{remark}\label{frobinus}
If $X$ is regular, then any $char(X)$-torsion element in $\Br'(X)$ is annihilated by the Frobenius morphism, which is finite locally free; therefore, it belongs to $\Br(X)$.
\end{remark}
\subsection{The norm residue homomorphism}
Fix a field $F$. Following Weibel \cite[Chapter 	III]{webl}, we denote by $K_2(F)$  the Milnor-Quillen algebraic $K_2$-group of $F$. If $E/F$ is a finite field extension,  one has the natural restriction map $i_{F/E}: K_2(F) \ra  K_2(E) $ induced by the inclusion $F \subset E$,  and the norm map $N_{E/F}: K_2(E) \ra  K_2(F)$ (also called the transfer map). The composition 
$$K_2(F) \overset{i_{F/E}}\lra  K_2(E) \overset{N_{E/F}}\lra  K_2(F)$$
is simply   multiplication by the extension degree $[E:F]$.

 Let $n \geq 0$ be an integer prime to  $char(F)$. Consider the Kummer exact sequence of \'etale sheaves on $\Spec(F)$
\begin{equation}
1 \lra \mu_n  \lra  \GG_m  \lra  \GG_m \lra 1.
\end{equation}
From the associated long exact sequence of cohomology, we get by Hilbert's 90 theorem  an isomorphism

$$ \partial: F^{*}/F^{*n} \overset{\simeq}\lra  H_\et^1(F,\mu_n).$$

Write $\mu_{n}^{\otimes 2} :=   \mu_n \otimes_\ZZ \mu_n $ and take the cup product, we get then the so called the  norm residue map (also known as the Galois symbol)

\begin{equation}\label{norm}
R^F_{2,n}: K_2(F)/nK_2(F)    \lra H_\et^2(F,\mu_{n}^{\otimes 2}),
\end{equation}
which takes a class $\{a,b\}+nK_2(F)$ to the cup product $\partial (a) \smile  \partial (b)$.

\begin{theorem}[Merkurjev-Suslin]\label{norm}
 The norm residue map $R^F_{2,n}$  is an isomorphism for all fields $F$ and  all $n$ prime to  $char(F)$. 
\end{theorem}
\begin{proof}
\cite[Theorem 11.5]{marks} (see also \cite[Theorem 8.6.5]{giltmas}).
\end{proof}
If $F$ contains a primitive $n$-th  root of unity $\zeta_n$, then we have an isomorphism
$$\theta^F_{\zeta_n}: H_\et^2(F,\mu_{n})  \overset{\simeq}\lra H_\et^2(F,\mu_n^{\otimes 2})$$
induced by the isomorphism $\theta_{\zeta_n} : \mu_n \overset{\simeq}\lra \mu_n  \otimes_\ZZ \mu_n$, $\eta \ra \eta \otimes  \zeta_n$, which depends on the primitive root $\zeta_n$. In particular, if  $X$ is a scheme over the field $F$,  then $\theta_{\zeta_n} $ induces an isomorphism of \'etale cohomology groups
$$\theta^X_{\zeta_n}:   H_\et^2(X,\mu_{n})  \overset{\simeq}\lra  H_\et^2(X,\mu_{n}^{\otimes 2}).$$

\subsection{Unramified extension of function fields}

Let $X$ be an integral scheme, and let $K(X)$ be its function field which is by definition the generic fiber, that is the  local ring $\shO_{X,\eta}$ of the  generic point $\eta$ of $X$. Let $L$ be a finite extension of $K(X)$. We recall that the normalization of $X$ in $L$ is an affine surjective morphism $f: Y \ra X$, where $Y$ is an integral scheme with function field $K(Y)=L$, and such that for all open affine $U$ of $X$, $\Gamma(f^{-1}(U),\shO_{Y})$ is the integral closure of $\Gamma(U,\shO_{X})$ in $L$. 

\begin{definition}\label{unramified}\cite[page 285]{katelang}
A  finite extension $L/K(X)$  is called  unramified over $X$, if the normalization of $X$ in $L$ is finite \'etale.
\end{definition}

Note that if $X$ is normal and $L/F(X)$ is separable, then the normalization of $X$ in $L$ is always finite (cf. \cite[Chapter I, Proposition 1.1]{MIL}). Further, if $X$ is a variety, this holds without the separability  assumption (cf. \cite[Chapter I, Remark 1.2]{MIL}).

Now suppose that $X$ is a connected normal locally noetherian scheme. The functor  fiber over $ \eta$
\begin{equation}
    (f: Y \ra X)  \lra  \shO_{Y,\eta'}
\end{equation}  with $f(\eta')=\eta$, 
 induces an equivalence of categories between connected finite \'etale covers of $X$ and finite separable extensions of $F(X)$ which are unramified over $X$ (cf. \cite[Page 292]{katelang}). This correspondence will allow us to construct the finite locally free cover  that kills the Brauer class $\alpha$ of Theorem \ref{mainthm}.

\section{Proof of Theorem \ref{mainthm}}

Notice first that since $X$ is smooth, the identification \eqref{identif} is guaranteed by the following two lemmas:
\begin{lemma}\label{mono}

For any  regular integral quasi-compact scheme  with function field $K(X)$ we have

\begin{itemize}
\item[(i)]  $\Pic(K(X)):=  H_\et^1(K(X),  \GG_{m})=0$.

\item[(ii)] The natural morphism $H^2_\et(X,  \GG_{m})     \ra   H_\et^2(K(X),  \GG_{m}) $
induced by the generic point $\eta: \Spec(K(X)) \ra X$ is injective.
\end{itemize}
\end{lemma}

\proof See \cite[Chapter III. Example 2.22]{MIL}.

\begin{lemma}\label{lift}

Let $X$ be a regular integral quasi-compact scheme with function field $K(X)$. For any non $char(X)$-torsion  class $\alpha  \in \Br'(X)$, there exists an integer $n_\alpha$ such that $\alpha$  lifts to a class $\beta \in H^2_\et(K(X),\mu_{n_\alpha})$.
\end{lemma}

\begin{proof}
By Lemma \ref{mono} the natural morphism
 $$H^2_\et(X,  \GG_{m})     \lra   H_\et^2(K(X),  \GG_{m}) $$
is injective. On the other hand, for any integer $n$ invertible on $X$, the long exact sequence of cohomology associated to the Kummer exact sequence for $K(X)$,  induces a short exact sequence
$$ 0 \lra \Pic(K(X)) {}_n\! \lra  H_\et^2(K(X), \mu_{n}) \lra  {}_n\! H_\et^2(K(X),  \GG_{m}) \lra 0.$$

By Lemma \ref{mono} $\Pic(K(X))=0$, thus we get an isomorphism
$$H_\et^2(K(X), \mu_{n}) \simeq  {}_n\! H_\et^2(K(X),  \GG_{m}).$$

The assertion follows immediately.
\end{proof} 

\begin{proof}[Proof of Theorem \ref{mainthm}]
Since $L/K$ is finite,  the compositum  $M=L.K(X)$ (inside some fixed algebraic closure of $K(X)$), is  unramified over $X$ (see \cite[Page 285]{katelang}), thus there is a finite \'etale cover $f: Y \ra X$ (which is the normalization of $X$ in $M$) with function field $K(Y)=M$. In fact, a simple argument on degree extensions shows that $M$ is equal to $L(X)$, which is  the function field of $X_L$.

Let $\alpha'$ be the image  of the class $\alpha$   in  $K_2(K(X))_p$, and put
\begin{equation}\label{alpha}
\alpha'_{|L(X)}=\beta \in K_2(L(X)).
\end{equation}

Since $\alpha'$ comes from $K_2(K)$, by the  following commutative diagram 
\begin{equation}
\centerline{\xymatrix{ 
  0 \ar[r]    & K_2(K)   \ar[rr]     & &   K_2(K(X)) &   \\
  0  \ar[r]    & K_2(L) \ar[u]^{N_{L/K}} \ar[rr]      & &  K_2(L(X))  \ar[u]^{N_{L/K}}  \ar@{=}[r] &    K_2(M) \ar[ul]  }}
\end{equation}
we see that $\alpha'=N_{L/K}(\alpha'')$ for  some $\alpha'' \in K_2(L(X))$. Consider the following commutative diagram

\begin{equation}
\centerline{\xymatrix{ 
 K_2(K(X))   \ar[drr]^{i_{K/L}}    & &     \\
 K_2(L(X)) \ar[u]^{N_{L/K}} \ar[rr]_{N_G}      & &  K_2(L(X))  }}
\end{equation}
 where $N_{G}=\sum_{g \in G} g$, and $G=\gal(L/K) $ (see \cite[page 6]{colliot}. It gives then
$$\beta=N_{L/K}(\alpha'')_{|L(X)}=N_{G}(\alpha'')=pn\alpha''$$
where $pn=[L:K]$ for some $n$, by assumption. This means  that $\alpha'$ belongs to the kernel of the natural map 
$$ K_2(K(X))_p \lra   K_2(M)_p.$$

Now     summing up,   we have a commutative diagram

\begin{equation}
\centerline{\xymatrix{  \Br'(X)_p  \ar[d]^{f^*}  \ar[r] &     H_\et^2(K(X),\mu_{p})   \ar[d]^{Res} \ar[rr]^ {\theta_{\zeta_p}^{K(X)}}   &&  H_\et^2(K(X),\mu_{p}^{\otimes 2})  \ar[d]^{Res}   \ar[rr]^{{R'}_{2,p}^{K(X)}}    &&     K_2(K(X))_p \ar[d] \\ \Br'(Y)_p  \ar[r]  &   
 H_\et^2(M,\mu_{p})  \ar[rr]^{\theta_{\zeta_p}^M}  &&  H_\et^2(M,\mu_{p}^{\otimes 2})   \ar[rr]^ {{R'}_{2,p}^M}  &&  K_2(M)_p  }}
\end{equation}
where   the maps  ${R'}_{2,p}^{K(X)}$ and ${R'}_{2,p}^M$ are inverses of  the  norm residue isomorphisms of Theorem \ref{norm} associated to  $K(X)$ and $M$, the two isomorphisms $\theta_{\zeta_p}^{K(X)}$ and $\theta_{\zeta_p}^M$ are induced by the isomorphism $\theta_{\zeta_p}: \mu_p \overset{\simeq}\lra \mu_p  \otimes_\ZZ \mu_p$ defined by $\zeta_p$, and $Res$ is the restriction map of group cohomology.

Since $X$ is a smooth variety, and  $f: Y \ra X$ is finite \'etale,  $Y$ is integral regular  and locally noetherian, hence by Lemmas \ref{mono}  and  \ref{lift} the  map $\Br'(Y)_p  \ra  H_\et^2(M,\mu_{p})$   is also injective. Now since $\alpha'$, the image of the class $\alpha$ in $K_2(K(X))_p$, maps to zero in  $K_2(M)_p$, we conclude by a  simple diagram chasing  that the image of $\alpha$ in $\Br'(Y)$ is zero. Therefore, by  Lemma  \ref{galois}, it lies in $\Br(X)$.
\end{proof}

\begin{remark}
One notes obviously, that since $X$  is smooth over $K$, it is then regular, and hence  the $char(K)$-torsions of $\Br'(X)$ are always in $\Br(X)$ by Remark \ref{frobinus}.

\end{remark}
\section{Proof of Theorem \ref{mainth2}}
Before proving Theorem \ref{mainth2}, we begin  by   the following lemma on the descent of the Brauer classes along base change.
\begin{lemma}\label{liftbase}

Let $X$ be a locally noetherian scheme over a field $K$, and let $L/K$ be a field extension. Let $X_L$ be the base change of $X$ to $L$. If $\alpha \in \Br'(X)$ such that $\alpha_{|{X_L}} \in \Br(X_L)$, then $\alpha \in \Br(X)$.

\end{lemma}

\begin{proof}
The extension $L/K$ can be written as a direct limit $L=\varinjlim L_i$ of finitely generated field extensions $L_i/K$, so then by  \cite[Chapter III,  Lemma 1.16]{MIL} one has $\Br'(X_L) = \varinjlim\Br'(X_{L_i})$. As for the Brauer group, we have also $\Br(X_L) = \varinjlim\Br(X_{L_i})$ (cf. \cite[2.2]{DJNG}).Therefore, we may assume  $L$  finitely generated. This being said, one can write $L$ as a finite extension of a purely transcendental extension of $K$. Therefore,  we can reduce the proof to the case when $L$ is finite or $L=K(x)$ for some variable $x$. The finite case follows by Lemma \ref{galois} since the natural morphism $X_L \ra X$ is finite locally free. Now suppose that $L=K(x)$, we have a factorization $\Spec(K(x)) \ra \AA^1_K \ra \Spec(K) $ where $\AA^1_K =\Spec(K[X])$, hence $X_L= (X \times_K \AA^1_K)\times_{\AA^1_K} K(x)$, and so $\alpha_{|X_L}$ would be the restriction of $\alpha_{|X \times_K \AA^1_K}$ to $X_L$. Since  $\alpha_{|X_L}$ comes from an Azumaya algebra, by \cite[II,  Theorem 2.1]{GR1} there is a nonempty subset $U \subset  \AA^1_K$ such that $\alpha_{|X \times_K U}$ is representable by an Azumaya algebra. Let $y: K(y) \ra U$ be a closed point. The residue field $K(y)$ is a finite extension of $K$, hence if we take $Y=(X \times_K U) \times_U K(y)$ then $Y \ra X$ is a finite locally free morphism with $\alpha_{|Y} \in \Br(Y)$, so Lemma \ref{galois} applies again.
\end{proof}

The following rigidity result of $K_2$-groups was proved by Colliot-Th\'el\`ene \cite[Theorem 1]{colliot} under some hypothesis on the base field, and by Suslin \cite[Theorem 5.8]{Suslin} for the general case.
\begin{theorem}\label{colliogalois}
Let $X$ be a smooth variety over a field $K$. For any Galois extension $L/K$ with Galois group  $G={\gal(L|K)}$, the natural morphism
$$K_2(K(X)) / K_2(K) \lra \bigg(K_2(L(X)) / K_2(L)\bigg)^{G}$$
is an isomorphism. 
\end{theorem}

We need also the follwing  local  result of Lim \cite[Theorem 1.1]{meng} for  norm maps of higher $K_{2n}$-groups  of number fields for $n \geq 1$. Colliot-Th\'el\`ene \cite[Lemma 2]{colliot}, and  Bak-Rehmann\cite[Theorem 1]{BakRehmann}   already proved   analogue statements for the  special case of $K_2$-groups.
\begin{theorem}\label{normfield}
Let  $L/K$  be a  finite extension of number fields. For all $n \geq 1$, we have an exact sequence
$$K_{2n}(L) \overset{N_{L/K}}\lra K_{2n}(K) \overset{i}\lra \bigoplus_{v \in S_r} K_{2n}(K_v)_{tors} \lra 0,$$
where $S_r$ is the set of real primes $v$  of $K$ ramified in $L$, $K_v$ is the completion of $K$ at $v$, and $i=(i_{K/K_v})_{v \in S_r}$ is the natural restriction  map.
\end{theorem}

\begin{proof}[Proof of Theorem \ref{mainth2}]

If $K$ contains a primitive $p$-th root of unity $\zeta_p$, then the identification \eqref{identif} is guaranteed. Take $L$ to be the splitting field over $K$ of the polynomial $f(x) = x^p - \omega$ for some $\omega \in K$ that is not a $p$-th power in $K$. This is a Galois extension, since $K$ has characteristic zero.  By construction, $f(x)$ is irreducible over $K$, and taking the subextension $K(\beta)$ for some root $\beta$ of $f(x)$, we find that $p = [K(\beta) : K]=[L : K]$. On the other hand, since $K$ is totally imaginary, the exact sequence  of Theorem \ref{normfield} shows that the norm map $N_{L/K}$ is surjective; hence   Theorem \ref{mainthm} applies.

Suppose now that $K$ does not contain a primitive $p$-th root of unity. Let $K' = K(\zeta_p)$ be the cyclotomic extension, and let $L$ be the splitting field over $K'$ as constructed above. Theorem \ref{colliogalois} tells us in particular that a class $\alpha' \in K_2(K(X))$ comes from $K_2(K)$ if and only if its image $\alpha'_{|K'(X)}$ in $K_2(K'(X))$ comes from $K_2(K')$. Therefore, by applying Lemma \ref{liftbase} to $K'$, we may assume that $K$ contains a primitive $p$-th root of unity, and proceeding as above, we  can again apply Theorem \ref{mainthm}.
\end{proof}

\section{Further Comments and Remarks}

We present some remarks on the assumptions stated in Theorem \ref{mainthm}, offering additional cases with alternative approaches to prove them:

\textbf{(a)} Regarding the assumption on the base field $K$, Theorem \ref{mainthm} can  be proved without any  restriction on the existence of $p$-th roots of unity, more precisely we have the following result.

\begin{proposition}
The statement of  Theorem \ref{mainthm}  holds  under the assumption that $L/K$ is finite of  degree $[L:K]=np$, with $n \wedge (p-1)=1$, and where neither  $K$ nor  $L$  necessary contains a primitive $p$-th root of unity.
\end{proposition}

\begin{proof}
We  follow an argument  of  Colliot-Th\'el\`ene \cite[page 9]{colliot}; Let $K'=K(\zeta_p)$, and let $M'=L.K'$ be  the compositum of  $L$ and $K'$ inside some fixed algebraic closure of $K$. The extension $M'$ is Galois over $K'$  with the same Galois group $G=\gal(L/K)$. Consider the following commutative diagram

\begin{equation}
\centerline{\xymatrix{  K_2(M')  \ar[d]_{N_{M'/K'}}  \ar[rr]^{N_{M'/L}} & &   K_2(L)     \ar[d]_{N_{L/K}}  \ar[rr]^{i_{L/M'}}  &&  K_2(M') \ar[d]_{N_{M'/K'}}     \\
 K_2(K') \ar@/_2pc/[rrrr]_{N_{G'}}   \ar[rr]_{N_{K'/K}} &&   K_2(K)      \ar[rr]_{i_{K/K'}}    &&   K_2(K')   }}
\end{equation}
where $N_{G'}=\sum_{g \in G'} g$, and $G'=\gal(K'/K)$ as in the proof of Theorem \ref{mainthm}. The two  compositions of horizontal arrows  are  multiplication by $[M':L]=[K':K]=p-1$. Now take a class $\alpha \in  K_2(K')^G$.  Since $\alpha$ is fixed by $G$, a diagram chasing shows that there is a class $\beta \in K_2(M') $ such that
$$N_{G}(\alpha)= \sum_{g \in G} g(\alpha)=(p-1)\alpha=  N_{M'/K'}(\beta).$$

On the other hand, from the  composition 
$$   \  K_2(K')   \lra   K_2(M')  \overset{N_{M'/K'}}\lra   K_2(K'),  $$
we see that $[M':K']\alpha=np\alpha=N_{M'/K'}(\beta')$, for some  $\beta \in  K_2(M') $. Since $np$ and $p-1$ are coprime,  by B\'ezout's  lemma  there exist integers $v$ and $w$ such that $N_{M'/K'}(v\beta'+w\beta)=\alpha$. Now, the image of the zero class of $ K_2(K(X)) / K_2(K)$  under the isomorphism of Theorem \ref{colliogalois}, comes exactly from $(K_2(K'))^G$. The remainder of the proof is then the same as in  the proof of Theorem \ref{mainthm},  one  need only  replace $K$ by $K'$ and apply lemma \ref{liftbase}.
\end{proof}

\textbf{(b)} If the primitive $p$-th root of unity in Theorem \ref{mainthm} is contained in $L$ and not necessarily in $K$, then we can replace condition $(ii)$ with a sufficient one, which is the existence of a finite extension $M/L$ with a surjective usual norm map $N: M^* \ra L^*$. Indeed, by   Lemma \ref{liftbase} and Theorem \ref{colliogalois} , we can replace $K$ by $L$ as in the proof of Theorem \ref{mainth2}. On the other hand, since $\zeta_p \in L$, then by \cite[Theorem 6.8]{webl} any element $\alpha \in K_2(L)$ can be written as a symbol $\{g, \zeta_p\}$ for some $g \in L^*$. Therefore, by the projection formula, we get
$$\alpha = \{g, \zeta_p\} = \{N(f), \zeta_p\} = N_{M/L}\{f, \zeta_p\}$$
for some $f \in M^*$, which gives the surjectivity of the norm map $N_{M/L}:  K_2(M) \ra  K_2(L)$.

\textbf{(c)} The surjectivity of the usual norm map for the extension $L(X)/K(X)$ does not work well for Theorem \ref{mainthm}. More precisely, suppose that $L/K$ is cyclic of prime degree $p$, such that the norm map $N: L(X)^* \ra K(X)^*$ is surjective. We can write $L = K(b)$, with $b^p = a \in K^*$, and we then have an exact sequence (see \cite[page 7]{colliot})

\begin{equation}\label{norminj}
0 \lra K(X)^*/N L(X)^*  \overset{\phi}\lra K_2(K(X))_p \lra K_2(L(X))_p,
\end{equation}
where $\phi$ is the map that sends $f \in K(X)^*$ to the class of the cyclic algebra $A^{\zeta}(a, f)$. Now, although the surjectivity of $N$ implies, by \cite[Proposition 6.6.2]{webl}, the surjectivity of $N_{L/K}: K_2(L(X)) \ra K_2(K(X))$, the exactness of the  sequence \eqref{norminj} means that in such a case, only trivial classes of $\Br'(X)_p$ are in $\Br(Y)$.

\end{document}